\documentclass[11pt]{amsart}
\usepackage{amssymb}

\newtheorem{theorem}{Theorem}[section]
\newtheorem{corollary}[theorem]{Corollary}
\newtheorem{lemma}[theorem]{Lemma}
\newtheorem{proposition}[theorem]{Proposition}
\theoremstyle{definition}
\newtheorem{definition}[theorem]{Definition}

\theoremstyle{remark}

\numberwithin{equation}{section}

\def\R {{\mathbb{R}}}

\def\3{{|\!|\!|}}

\def\d#1#2{{\rm dec}({\bf\Sigma}^0_{#1},{\bf\Delta}^0_{#2})}
\def\jr#1#2{^{-1}{\bf\Sigma}^0_{#1}\subseteq{\bf\Sigma}^0_{#2}}

\begin{document}

\title{Decomposing functions of Baire class $2$ on Polish spaces}
\author{Longyun Ding, Takayuki Kihara, Brian Semmes, and Jiafei Zhao}
\address{(Longyun Ding and Jiafei Zhao) School of Mathematical Sciences and LPMC, Nankai University, Tianjin, 300071, P.R.China}
\email{dinglongyun@gmail.com (Longyun Ding), 294465868@qq.com (Jiafei Zhao)}
\address{(Takayuki Kihara) Graduate School of Informatics, Nagoya University, Nagoya, 464-8601, Japan}
\email{kihara@i.nagoya-u.ac.jp}
\address{(Brian Semmes) StudyLab Language School, Moscow, Nikolskaya St. 10, Nikolskaya Plaza Office Centre, Russian Federation}
\email{brian@studylab.ru}
\thanks{The Research of the first author was partially supported by the National Natural Science Foundation of China (Grant No. 11725103). The second author was partially supported by JSPS KAKENHI Grant
17H06738 and 15H03634, and also thanks JSPS Core-to-Core Program (A. Advanced Research Networks) for supporting the research.}

\subjclass[2010]{Primary 03E15, 54H05, 26A21}

%\date{\today}

\begin{abstract}
We prove the Decomposability Conjecture for functions of Baire class $2$ on a Polish space to a separable metrizable space. This partially answers an important open problem in descriptive set theory.
\end{abstract}
\maketitle

\section{Introduction}

In descriptive set theory, the study of decomposability of Borel functions originated by a famous question asked by Luzin around a century ago: Is every Borel function decomposable into countably many continuous functions? This question was answered negatively. Many counterexamples appeared in the literature (cf. \cite{keldys,kuratowski}) show that, even a function of Baire class $1$ is not necessarily decomposable. Among these counterexamples, the Pawlikowski function $P:(\omega+1)^\omega\to\omega^\omega$ stands in an important position. Indeed, Solecki \cite{solecki1} proved that:
\begin{quote}
Let $X,Y$ be separable metrizable spaces with $X$ analytic, and let $f:X\to Y$ be of Baire class $1$. Then $f$ is not decomposable into countably many continuous functions iff $P\sqsubseteq f$, i.e., there exists embeddings $\phi:(\omega+1)^\omega\to X$ and $\psi:\omega^\omega\to Y$ such that $\psi\circ P=f\circ\phi$.
\end{quote}
Later, Pawlikowski and Sabok~\cite{PS} generalized this theorem onto all Borel functions from an analytic space to a separable metrizable space. Motto Ros~\cite[Lemma 5.6]{motto_ros} also gave an elegant proof for all functions of Baire class $n$ with $n<\omega$.

A natural generalization of Luzin's question is to replace continuous functions with ${\bf\Sigma}^0_\gamma$-measurable functions. We write $f\in{\rm dec}({\bf\Sigma}^0_\gamma)$ if there exists a partition $(X_k)$ of $X$ with each $f\upharpoonright X_k$ is ${\bf\Sigma}^0_\gamma$-measurable; and also write $f\in\d{\gamma}{\delta}$ if such a partition can be a sequence of ${\bf\Delta}^0_\delta$ subsets of $X$. It is trivial to see that, for $\delta\ge\gamma$, $f\in\d\gamma\delta$ implies the ${\bf\Sigma}^0_\delta$-measurability of $f$. It is also well known that, for any ${\bf\Sigma}^0_\delta$-measurable function $f$ with $\delta>\gamma$, we have $f\in{\rm dec}({\bf\Sigma}^0_\gamma)\iff f\in\d{\gamma}{\delta+1}$ (cf. \cite[Proposition 4.5]{motto_ros}).

A slightly more finer notion of Baire hierarchy was essentially introduced by Jayne \cite{jayne} for studying the Banach space of functions of Baire class $\alpha$. A function $f:X\to Y$ is called a ${\bf\Sigma}_{\alpha,\beta}$ function (or more precisely denoted by $f\jr\beta\alpha$) if the preimage $f^{-1}(A)$ is ${\bf\Sigma}^0_\alpha$ in $X$ for every ${\bf\Sigma}^0_\beta$ subset $A$ of $Y$. The following theorem discovers a deep connection between this notion and decomposability:
\begin{theorem}[Jayne-Rogers \cite{JR}]\label{thm:JR-theorem}
Let $X,Y$ be separable metrizable spaces with $X$ analytic, and let $f:X\to Y$. Then
$$f\jr22\iff f\in\d12.$$
\end{theorem}

This theorem was generalized in \cite{KMS} to the case that $X$ is an absolute Souslin-$\mathcal F$ set and $Y$ is an arbitrary regular topological space.

%One can expect that the Jayne-Rogers Theorem can be generalized to all Borel ranks.
It is conjectured that the Jayne-Rogers Theorem can be extended to all finite Borel ranks as follows:

\medskip

\noindent
{\bf The Decomposability Conjecture} (cf. \cite{andretta,motto_ros,PS}). Let $X,Y$ be separable metrizable spaces with $X$ analytic, and let $f:X\to Y$. Then for $n\ge 2$ we have
$$f\jr{n}n\iff f\in\d1n.$$
Furthermore, for $2\le m\le n$ we have
$$f\jr{m}n\iff f\in\d{n-m+1}n.$$

%For the case of infinite ordinals, note that all functions of Baire class $1$ are ${\bf\Sigma}_{\omega,\omega}$ functions. So the Pawlikowski function $P$ is ${\bf\Sigma}_{\omega,\omega}$ function but $P\notin{\rm dec}({\bf\Sigma}^0_1)$, and therefore, we cannot replace $n$ with $\omega$ in the above conjecture.
%Nevertheless, there is a correct way to generalize the Decomposability Conjecture to all infinite ordinals (see \cite[Section 4]{GKN}).
This conjecture was further generalized to The Full Decomposability Conjecture (see \cite[Section 4]{GKN}) which covers all infinite Borel ranks.
Motto Ros presented an equivalent condition of the decomposability conjecture (see \cite[Conjecture 6.1]{motto_ros}). Another interesting equivalent condition with some extra restrictions on spaces and on relation between $m,n$, concerning computability on Borel codes from $A$ to $f^{-1}(A)$, was given by Kihara in \cite{kihara}. Most recently, Gregoriades-Kihara-Ng \cite{GKN} proved
$$f\jr{m}n\ \Longrightarrow\ f\in\d{n-m+1}{n+1}.$$

It is clear that the case $m=n=2$ in the decomposability conjecture is just the Jayne-Rogers Theorem. Remarkable progress is due to Semmes, the third author of this article. In his Ph.D. thesis \cite{semmes}, Semmes proved the case $m\le n=3$ for functions $f:\omega^\omega\to\omega^\omega$. In his proof, many kinds of games for characterizing Borel functions were widely applied.
From the viewpoint of Jayne's work \cite{jayne} in functional analysis, the zero-dimensionality constraint on Semmes' theorem was strongly desired to be removed. In this article, we generalize Semmes' theorem to arbitrary Polish spaces:

\begin{theorem}\label{main}
The decomposability conjecture is true for the case that $X$ is Polish space and $m\le n=3$.
\end{theorem}

It is worth noting that in our proof, no game for Borel functions are involved. This is the key point that this proof can be extended to all Ploish spaces. This theorem consists of two cases: (a) $m=2,n=3$, and (b) $m=n=3$. We will prove them in sections 3 and 4 respectively.

Following the outline of Semmes' proof, the proof appearing in this article was developed by the first and the forth authors. Almost at the same time, the second author independently gave a detailed exposition of Semmes' strategy. He also asserted that the use of games for Borel functions is misleading, and emphasized the use of finite injury priority argument instead. Soon after reading it, Motto Ros pointed out that the same argument in the second author's proof also works well, with some minor modifications, for arbitrary Polish spaces.
%Concerning the third author's thesis, in a note by Carroy on Louveau's version of Semmes' proof, some gaps were indicated. We were suggested by a number of experts to make clear how to avoid these gaps. Unfortunately, though the ideas behind these 4 versions of proofs are similar, the definitions, notations and the ways to organize the proofs contained in them are totally deferent. This leads us failed to find the corresponding positions of these gaps in the current proof.

The authors would like to thank Rapha\"el Carroy and Luca Motto Ros for helpful suggestions. The first author is grateful to Slawomir Solecki for his suggestions and encouragement.
The second author also would like to thank Vassilios Gregoriades, Tadatoshi Miyamoto, and Yasuo Yoshinobu for valuable discussions.

\section{Preliminaries}

{\sl All topological spaces considered in this article are separable metrizable.} For any subset $A$ of a topological space $X$, we denote by $\overline{A}$ the closure of $A$ in $X$ and denote $A^c=X\setminus A$ for brevity.

We recall some basic notations. A topological space is called a {\it Polish space} if it is separable and completely metrizable, and is called an {\it analytic space} if it is homeomorphic to an analytic subset of a Polish space. Given a separable metrizable space $X$, Borel sets of $X$ can be analyzed into Borel hierarchy, consisting of ${\bf\Sigma}^0_\xi$, ${\bf\Pi}^0_\xi$ subsets for $1\le\xi<\omega_1$. As usual, we denote ${\bf\Delta}^0_\xi={\bf\Sigma}^0_\xi\cap{\bf\Pi}^0_\xi$.

Let $X,Y$ be two separable metrizable spaces, and $f:X\to Y$. We say $f$ is ${\bf\Sigma}^0_\alpha$-measurable if $f^{-1}(U)\in{\bf\Sigma}^0_\alpha$ for every open set $U\subseteq Y$. For the definition of the Baire classes of functions, one can see \cite[(24.1)]{kechris}. It is well known that a function is of Baire class $\xi$ iff it is ${\bf\Sigma}^0_{\xi+1}$-measurable (cf. \cite[(24.3)]{kechris}).

In the section of introduction, we already presented notion of ${\bf\Sigma}_{\alpha,\beta}$ functions, $f\jr\beta\alpha$ and $\d\gamma\delta$. The following proposition give some well known properties which will be used again and again in the rest of this article.

\begin{proposition}[folklore]\label{proposition}
Let $X,Y$ be two separable metrizable spaces, and let $f:X\to Y$. Then the following are equivalent:
\begin{enumerate}
\item[(i)] $f\in\d\gamma\delta$.
\item[(ii)] There exists a sequence $(A_n)$ of ${\bf\Sigma}^0_\delta$ subsets with $X=\bigcup_nA_n$ such that every $f\upharpoonright A_n$ is ${\bf\Sigma}^0_\gamma$-measurable.
\item[(iii)] There exists a sequence $(A_n)$ of ${\bf\Sigma}^0_\delta$ subsets with $X=\bigcup_nA_n$ such that every $f\upharpoonright A_n\in\d\gamma\delta$.
\end{enumerate}
\end{proposition}

\begin{proof}
(i)$\Rightarrow$(ii) and (ii)$\Rightarrow$(iii) are trivial. We only prove (iii)$\Rightarrow$(i).

For every $n<\omega$, since $A_n\in{\bf\Sigma}^0_\delta$, we can choose a sequence $(B_n^m)_{m<\omega}$ of ${\bf\Delta}^0_\delta$ sets such that $A_n=\bigcup_mB_n^m$. Moreover, since
$$f\upharpoonright A_n\in\d\gamma\delta,$$
there exist two sequences $(C_n^k)_{k<\omega},(D_n^k)_{k<\omega}$ of ${\bf\Sigma}^0_\delta$ sets with
$$A_n\subseteq\bigcup_kC_n^k,\quad A_n\cap C_n^k=A_n\setminus D_n^k$$
such that each $f\upharpoonright(C_n^k\cap A_n)$ is ${\bf\Sigma}^0_\gamma$-measurable. Note that $B_n^m\cap C_n^k=B_n^m\setminus D_n^k\in{\bf\Delta}^0_\delta$, and $f\upharpoonright(B_n^m\cap C_n^k)$ is ${\bf\Sigma}^0_\gamma$-measurable for all $n,k,m<\omega$.

Let $(K_l)_{l<\omega}$ be an enumeration of all $B_n^m\cap C_n^k$,\, $n,k,m<\omega$. Then
$$\bigcup_lK_l=\bigcup_{n,k,m}(B_n^m\cap C_n^k)=X.$$
For each $l<\omega$, put $K'_l=K_l\setminus(\bigcup_{i<l}K_i)$. Then the sequence $(K_l')_{l<\omega}$ of ${\bf\Delta}^0_\delta$ subsets witnesses that $f\in\d\gamma\delta$.
\end{proof}

\section{The decomposability conjecture for $m=2,n=3$}

We prove Theorem~\ref{main} for $m=2,n=3$ in this section, and for $m=n=3$ in the next section. The following lemma is the key tool for proving the main theorem of this section, just like the role of Lemma 4.3.3 in \cite{semmes}.

\begin{lemma}\label{PDtoVAF}
Let $X,P$ be two separable metrizable spaces, and let $D\subseteq X$, $h:D\to P$ a function of Baire class $2$. Let $\mathcal B_P$ be a countable topological basis of $P$, and for each $V\in\mathcal B_P$, let $\mathcal G_V$ be a countable class of subsets of $D$ such that
$$h^{-1}(V)=\bigcup\mathcal G_V.$$
If $h\notin\d23$, then there exist $V\in\mathcal B_P, G\in\mathcal G_V$, and a closed set $F\subseteq\overline D$ satisfying:
\begin{enumerate}
\item[(a)] For any open set $U$ with $F\cap U\ne\emptyset$,
$$h\upharpoonright(h^{-1}(\overline{V}^c)\cap F\cap U)\notin\d23;$$
\item[(b)] $G\cap F$ is dense in $F$;
\item[(c)] $F\cap D\ne\emptyset$.
\end{enumerate}
\end{lemma}

\begin{proof}
Let $\{U_k:k<\omega\}$ be a topological basis of $X$. For any $V\in\mathcal B_P$ and any closed subset $F\subseteq X$, we denote
$$\Gamma_V(F)=\{k<\omega:h\upharpoonright(h^{-1}(\overline{V}^c)\cap F\cap U_k)\in\d23\},$$
$$\Theta_V(F)=\{x\in F:\forall k<\omega(x\in U_k\Rightarrow k\notin\Gamma_V(F))\}.$$
It is trivial to see that $\Theta_V(F)\subseteq\overline D$ is closed.

For any $G\in\mathcal G_V$, we define closed set $F_{V,G}^\alpha$ for $\alpha<\omega_1$ as follows:
$$F_{V,G}^0=X,$$
$$F_{V,G}^{\alpha+1}=\overline{G\cap\Theta_V(F_{V,G}^\alpha)},$$
$$F_{V,G}^\lambda=\bigcap_{\alpha<\lambda}F_{V,G}^\alpha,\quad\mbox{ for limit ordinal }\lambda.$$
Since $X$ is second countable, there exists a $\xi<\omega_1$ such that $F_{V,G}^\alpha=F_{V,G}^\xi$ for each $V,G$ and $\alpha\ge\xi$.

If there exist $V\in\mathcal B_P, G\in\mathcal G_V$ such that $F_{V,G}^\xi\ne\emptyset$, then $V,G$, and $F=F_{V,G}^\xi$ fulfil clauses (a) and (b). Set $U=X$ in (a), we can see (c) is also fulfilled.

Assume for contradiction that, for any $V\in\mathcal B_P, G\in\mathcal G_V$, we have $F_{V,G}^\xi=\emptyset$. For $\alpha<\xi$ and $k\in\Gamma_V(F_{V,G}^\alpha)$, put
$$H_{V,G,k}^\alpha=F_{V,G}^\alpha\cap U_k.$$
Note that
$$h\upharpoonright(h^{-1}(\overline{V}^c)\cap H_{V,G,k}^\alpha)\in\d23.$$

Now define a subset $H$ of all $x$ satisfying that, there exist $V_1,V_2\in\mathcal B_P$ with $\overline{V_1}\cap\overline{V_2}=\emptyset$, and for $i=1,2$, there exist $G_i\in\mathcal G_{V_i}$, $\alpha_i<\xi$, and $k_i\in\Gamma_{V_i}(F_{V_i,G_i}^{\alpha_i})$ such that $x\in H_{V_1,G_1,k_1}^{\alpha_1}\cap H_{V_2,G_2,k_2}^{\alpha_2}$. Since
$$h\upharpoonright(h^{-1}(\overline{V_i}^c)\cap H_{V_1,G_1,k_1}^{\alpha_1}\cap H_{V_2,G_2,k_2}^{\alpha_2})\in\d23\quad(i=1,2),$$
and $h^{-1}(\overline{V_i}^c)$ is ${\bf\Sigma}^0_3$ in $D$ for $i=1,2$, by Proposition~\ref{proposition}, we have
$$h\upharpoonright(D\cap H_{V_1,G_1,k_1}^{\alpha_1}\cap H_{V_2,G_2,k_2}^{\alpha_2})\in\d23.$$
Therefore, $h\upharpoonright(D\cap H)\in\d23$.

For any $x\in D$, $V\in\mathcal B_P$, and $G\in\mathcal G_V$ with $x\in G$, we claim that there exist $\alpha<\xi$ and $k\in\Gamma_V(F_{V,G}^\alpha)$ such that $x\in H_{V,G,k}^\alpha$. There is unique $\alpha<\xi$ such that $x\in(F_{V,G}^\alpha\setminus F_{V,G}^{\alpha+1})$. Note that $x\notin(F\setminus\overline{G\cap F})$ for any closed set $F\subseteq X$, so $x\notin(\Theta_V(F_{V,G}^\alpha)\setminus F_{V,G}^{\alpha+1})$. From the definition of $\Theta_V(F_{V,G}^\alpha)$, we can find a $k\in\Gamma_V(F_{V,G}^\alpha)$ such that $x\in U_k$. Then we have $x\in F_{V,G}^\alpha\cap U_k=H_{V,G,k}^\alpha$.

In the end, we consider $h\upharpoonright(D\setminus H)$. First, for any $x\in(D\setminus H)$, if $x\in H_{V,G,k}^\alpha$ for some $V\in\mathcal B_P$, $G\in\mathcal G_V$, $\alpha<\xi$, and $k\in\Gamma_V(F_{V,G}^\alpha)$, we claim that $h(x)\in\overline{V}$. If not, we can find a $V'\in\mathcal B_P$ such that $h(x)\in V'$ and $\overline{V}\cap\overline{V'}=\emptyset$. Since $h^{-1}(V')=\bigcup\mathcal G_{V'}$, we can find an $G'\in\mathcal G_{V'}$ such that $x\in G'$. Hence $x\in H_{V',G',k'}^{\alpha'}$ for some $\alpha'<\xi$ and $k'\in\Gamma_{V'}(F_{V',G'}^{\alpha'})$, contradicting $x\notin H$. Secondly, let $d$ be a compatible metric on $P$. For any $n<\omega$, let
$$(V^n_m,G^n_m,k^n_m,\alpha^n_m)_{m<\omega}$$
be an enumeration of all $(V,G,k,\alpha)$ with ${\rm diam}(\overline{V})\le 1/n$, $G\in\mathcal G_V$, $\alpha<\xi$, and $k\in\Gamma_V(F_{V,G}^\alpha)$. Denote
$$H^n_m=H_{V^n_m,G^n_m,k^n_m}^{\alpha^n_m}.$$
For any $x\in D$, we can find a $V\in\mathcal B_P$ with ${\rm diam}(\overline{V})\le 1/n$ such that $h(x)\in V$ and a $G\in\mathcal G_V$ with $x\in G$. Hence $x\in H_{V,G,k}^\alpha$ for some $\alpha<\xi$ and $k\in\Gamma_{V}(F_{V,G}^\alpha)$.  It follows that $D\subseteq\bigcup_mH^n_m$. Put $K^n_m=H^n_m\setminus\bigcup_{k<m}H^n_k$ for each $m$. Then $(K^n_m)_{m<\omega}$ is a sequence of pairwise disjoint ${\bf\Delta}^0_2$ sets. Fix a $y^n_m\in\overline{V^n_m}$ for each $m$. Define $g_n(x)=y^n_m$ for all $x\in K^n_m$. Then $g_n$ is of Baire class 1. Furthermore, we have $d(g_n(x),h(x))\le 1/n$ for all $x\in(D\setminus H)$. So $(g_n\upharpoonright(D\setminus H))_{n<\omega}$ uniformly converges to $h\upharpoonright(D\setminus H)$. It follows that $h\upharpoonright(D\setminus H)$ is of Baire class 1 also (see \cite[(24.4) i)]{kechris}).

Note that $H$ is an $F_\sigma$ set from its definition. So $D\setminus H$ is $G_\delta$ in $D$, and hence $h\in\d23$. A contradiction!
\end{proof}

In the rest of this section, we fix $X$ be a Polish space, $Y$ a separable metrizable space, and $f:X\to Y$ a ${\bf\Sigma}^0_3$-measurable function.

\begin{definition}
Let $\mathcal F=\langle F_0,\cdots,F_k\rangle$ be a finite sequence of closed sets of $X$ with $F_0\supseteq\cdots\supseteq F_k$, $U$ an open subset of $X$, and let $P\subseteq Y$.
\begin{enumerate}
\item[(i)] If $k=0$, i.e., $\mathcal F=\langle F_0\rangle$, then we say $\mathcal F$ is $P$-{\bf sharp} in $U$ if $U\cap F_0\ne\emptyset$, and for any open set $U'\subseteq U$ with $U'\cap F_0\ne\emptyset$, we have
$$f\upharpoonright(f^{-1}(P)\cap F_0\cap U')\notin\d23.$$
We also say $F_0$ itself is $P$-sharp in $U$ for brevity.
\item[(ii)] If $k>0$, then we say $\mathcal F$ is $P$-{\bf sharp} in $U$ if $F_k$ is $P$-sharp in $U$, and for any open set $U'\subseteq U$ with $U'\cap F_k\ne\emptyset$, $\mathcal F\upharpoonright k$ is $P$-sharp in some open set $U''\subseteq U'$.
\end{enumerate}
\end{definition}

A similar notion named $\delta$-$\sigma$-good was presented in \cite{semmes}. The following propositions are trivial, we omit the proofs.

\begin{proposition}\label{sharp1}
Suppose $\mathcal F=\langle F_0,\cdots,F_k\rangle$ is $P$-sharp in $U$. Then for any open set $U'\subseteq U$ with $U'\cap F_k\ne\emptyset$, we have $\mathcal F$ is $P$-sharp in $U'$.
\end{proposition}

\begin{proposition}\label{sharp2}
Suppose $\mathcal F$ is $P$-sharp in $U$. Then for any $m<{\rm lh}(\mathcal F)$, $\mathcal F\upharpoonright m$ is $P$-sharp in some open set $U'\subseteq U$.
\end{proposition}

The following lemma is modified from \cite[Lemma 4.3.6]{semmes}.

\begin{lemma}\label{sharp}
Suppose $\mathcal F=\langle F_0,\cdots,F_k\rangle$ is $P$-sharp in $U$. Let $(C_l)_{l<m}$ be a sequence of pairwise disjoint closed subsets of $P$. Then there exist at most $k+1$ many $l$ such that $\mathcal F$ is not $P\setminus C_l$-sharp in any open set $U'\subseteq U$.
\end{lemma}

\begin{proof}
We begin with $k=0$. Without loss of generality, suppose there exists an $l<m$, say, $l=0$, such that $F_0$ is not $P\setminus C_0$-sharp in $U$. Then there exists an open set $U_0\subseteq U$ with $U_0\cap F_0\ne\emptyset$ such that
$$f\upharpoonright(f^{-1}(P\setminus C_0)\cap F_0\cap U_0)\in\d23.$$
Assume for contradiction that there exists $l\ne 0$ such that $F_0$ is not $P\setminus C_l$-sharp in $U_0$, then there is an open set $U_l\subseteq U_0$ with $U_l\cap F_0\ne\emptyset$ such that
$$f\upharpoonright(f^{-1}(P\setminus C_l)\cap F_0\cap U_l)\in\d23.$$
Since $C_0$ and $C_l$ are disjoint closed subsets of $P$, Proposition~\ref{proposition} gives
$$f\upharpoonright(f^{-1}(P)\cap F_0\cap U_l)\in\d23,$$
contradicting that $F_0$ is $P$-sharp in $U$.

For $k>0$, assume that we have proved for all $k'<k$. Since $F_k$ is $P$-sharp in $U$, from the arguments for $k=0$ above, we may assume that there is an open set $U_0\subseteq U$ with $U_0\cap F_k\ne\emptyset$ such that $F_k$ is $P\setminus C_l$-sharp in $U_0$ for any $l\ne 0$. Assume for contradiction that there are $k+1$ many $l\ne 0$, say, $l=1,\cdots,k,k+1$, such that $\mathcal F$ is not $P\setminus C_l$-sharp in any open set $U'\subseteq U_0$. Particularly, $\mathcal F$ is not $P\setminus C_1$-sharp in $U_0$, so there exists an open set $U_1\subseteq U_0$ with $U_1\cap F_k\ne\emptyset$ such that $\mathcal F\upharpoonright k$ is not $P\setminus C_1$-sharp in any open set $U'\subseteq U_1$. Similarly, we can find a sequence of open sets $U_{k+1}\subseteq U_k\subseteq\cdots\subseteq U_1\subseteq U_0$ such that $U_l\cap F_k\ne\emptyset$ and $\mathcal F\upharpoonright k$ is not $P\setminus C_l$-sharp in any $U'\subseteq U_l$ for $0<l\le k+1$. By Definition of $P$-sharp, there is a open set $U^*\subseteq U_{k+1}$ such that $\mathcal F\upharpoonright k$ is $P$-sharp in $U^*$, contradicting the induction hypothesis.
\end{proof}

Let $\ulcorner\cdot,\cdot\urcorner$ be the bijection: $\omega\times\omega\to\omega$ as following:
$$\ulcorner 0,0\urcorner=0,$$
$$\ulcorner 0,j+1\urcorner=\ulcorner j,0\urcorner+1,$$
$$\ulcorner i+1,j-1\urcorner=\ulcorner i,j\urcorner+1.$$

Denote
$$\Omega=\{z\in 2^\omega:\exists i\exists^\infty j(z(\ulcorner i,j\urcorner)=1)\}.$$
It is well known that $\Omega$ is ${\bf\Sigma}^0_3$-complete subset of $2^\omega$.

For any $z\in 2^\omega$ and $l<\omega$, we call sequence
$$z\upharpoonright(\ulcorner 0,l\urcorner+1),\;z\upharpoonright(\ulcorner 1,l-1\urcorner+1),\;\cdots,\;z\upharpoonright(\ulcorner l,0\urcorner+1)$$
the $l$-th {\bf diagonal} of $z$, and call $z\upharpoonright(\ulcorner l,0\urcorner+1)$ the end of $l$-th diagonal. For $s\in 2^{<\omega}$, we denote ${\rm lh}(s)=i$ the length of $s$. If $s\subseteq z$ and ${\rm lh}(s)=\ulcorner i,j\urcorner+1$, then $s$ is in $(i+j)$-th diagonal. Moreover, the $l$-th diagonal of $z$ is also named the $l$-th diagonal of $s$ when $\ulcorner l,0\urcorner<{\rm lh}(s)$.

For $s\ne\emptyset$, let ${\rm lh}(s)=\ulcorner i,j\urcorner+1$. We denote ${\rm row}(s)=i,{\rm col}(s)=j$. If $i+j>0$, we call $s\upharpoonright(\ulcorner i+j-1,0\urcorner+1)$ {\bf the end of the last diagonal} of $s$, denoted by $u(s)$. If $j>0$, we call $s\upharpoonright(\ulcorner i,j-1\urcorner+1)$ {\bf the left neighbor} of $s$, denoted by $v(s)$.

For proving the following theorem, we need an order $\preceq$ on $2^{<\omega}$ define by
$$t\preceq s\iff{\rm lh}(t)<{\rm lh}(s)\mbox{ or }({\rm lh}(t)={\rm lh}(s),t\le_{\rm lex}s),$$
where $\le_{\rm lex}$ is the usual lexicographical order. We also denote $t\prec s$ when $t\preceq s$ but not $t=s$. Denote
$$s_M={\textstyle\max_\preceq}\{t:t\prec s^\smallfrown 0\}.$$

\begin{theorem}\label{3to2}
Let $X$ be a Polish space, $Y$ a separable metrizable space, and let $f:X\to Y$. Then
$$f^{-1}{\bf\Sigma}^0_2\subseteq{\bf\Sigma}^0_3\iff f\in\d23.$$
\end{theorem}

\begin{proof}
The ``$\Leftarrow$'' part is trivial, we only prove the ``$\Rightarrow$'' part.

Assume for contradiction that $f\notin\d23$. We will define a continuous embedding $\psi:2^\omega\to X$ and an open set $O\subseteq Y$ such that
$$\psi^{-1}(f^{-1}(O))=\Omega.$$
Thus $f^{-1}(Y\setminus O)$ is ${\bf\Pi}^0_3$-complete subset of $X$, contradicting $f^{-1}{\bf\Sigma}^0_2\subseteq{\bf\Sigma}^0_3$.

For any open set $V\subseteq Y$, since $f^{-1}(V)$ is ${\bf\Sigma}^0_3$, we can fix a system of open set $W^m_n(V)\subseteq X\;(m,n<\omega)$ with
$$f^{-1}(V)=\bigcup_m\bigcap_nW^m_n(V).$$
Denote $G^m(V)=\bigcap_nW^m_n(V)$. For $G=G^m(V)$, we also denote $W_n(G)=W^m_n(V)$.

For $s\in 2^{<\omega}$ with $s\ne\emptyset$, let ${\rm lh}(s)=\ulcorner i,j\urcorner+1$. We say $s$ is an {\bf inheritor} if $j>0$ and $s(\ulcorner k,i+j-k\urcorner)=0$ for any $k<i$, otherwise we say $s$ is an {\bf innovator}. Note that $s$ is always an inheritor if $i=0$, $j>0$, and is always an innovator if $j=0$.

Fix a compatible metric $d$ on $X$ with $d\le 1$. We will inductively construct, for each $s\in 2^{<\omega}$, an open set $V_s$ of $Y$, a $G_\delta$ set $G_s$ of $X$, a closed set $F_s$ of $X$, an open set $U_s$ of $X$, and a sequence of open sets $(U_s^w)_{s\preceq w\prec s^\smallfrown 0}$ of $X$ satisfying the following:
\begin{enumerate}
\item[(0)] ${\rm diam}(\overline{U_s})\le 2^{-{\rm lh}(s)}$, $U_{s^\smallfrown 0}\cap U_{s^\smallfrown 1}=\emptyset$, $\overline{U_{s^\smallfrown 0}}\subseteq U_s^w$, $\overline{U_{s^\smallfrown 1}}\subseteq U_s^w$;
\item[(1)] $V_{s^\smallfrown 0}=V_{s^\smallfrown 1}$, $G_{s^\smallfrown 0}=G_{s^\smallfrown 1}$, $F_{s^\smallfrown 0}=F_{s^\smallfrown 1}$;
\item[(2)] for $s,t\in 2^{<\omega}$, we have $V_s=V_t$ or $\overline{V_s}\cap\overline{V_t}=\emptyset$;
\item[(3)] there exist $m<\omega$ such that $G_s=G^m(V_s)$;
\item[(4)] if ${\rm col}(s)>0$, then $F_{s^\smallfrown 0}=F_{s^\smallfrown 1}\subseteq F_s$;
\item[(5)] $F_s\cap U_s^w\ne\emptyset$ for each $w$;
\item[(6)] $U_s=U_s^s$, and $U_s^{w_1}\supseteq U_s^{w_2}$ for $w_1\preceq w_2$;
\item[(7)] $G_s\cap F_s$ is dense in $F_s$;
\item[(8)] if $s$ is an inheritor, then we have
$$V_s=V_{v(s)},\quad G_s=G_{v(s)},\quad F_s=F_{v(s)};$$
furthermore, (a) if $s^\smallfrown 0$ is also an inheritor, then $U_s^w\cap F_{u(s)}\ne\emptyset$ for each $w$; (b) if $s^\smallfrown 0$ is an innovator, then $U_s\subseteq W_n(G_s)$ for all $n<{\rm lh}(s)$;
\item[(9)] if $s$ is an innovator, then $U_s\cap F_s\cap G^m(V_t)=\emptyset$ for all $m<{\rm lh}(s)$ and all ${\rm lh}(t)<{\rm lh}(s)$;
\item[(10)] by letting $P_s=Y\setminus\bigcup_{t\preceq s}\overline{V_t}$,
$$\mathcal F_s=\langle F_{s\upharpoonright({\rm lh}(s)-{\rm row}(s))},\cdots,F_{s\upharpoonright({\rm lh}(s)-1)},F_s\rangle,$$
for $t\preceq s\prec t^\smallfrown 0$, we have
    \\(a) if $t^\smallfrown 0$ is an innovator, then $\mathcal F_t$ is $P_s$-sharp in $U_t^s$;
    \\(b) if $t^\smallfrown 0$ is an inheritor, then $\mathcal F_{u(t^\smallfrown 0)}$ is $P_s$-sharp in $U_t^s$.
\end{enumerate}

When we complete the construction, for any $z\in 2^\omega$, we set $\psi(z)$ to be the unique element of $\bigcap_kU_{z\upharpoonright k}$. From (0) and (6), we see that $\psi$ is a continuous embedding from $2^\omega$ to $X$. Put
$$O=\bigcup_{t\in 2^{<\omega}}V_t.$$

If $z\in\Omega$, let $i_0$ be the least $i$ such that there are infinitely many $j$ with $z(\ulcorner i,j\urcorner)=1$. Then there is $J_0<\omega$ such that $z(\ulcorner i,j\urcorner)=0$ for all $i<i_0$ and all $j>J_0$. Hence for any $j>J_0$, $z\upharpoonright(\ulcorner i_0,j\urcorner+1)$ is an inheritor. Denote
$$V=V_{z\upharpoonright(\ulcorner i_0,J_0\urcorner+1)},\quad G=G_{z\upharpoonright(\ulcorner i_0,J_0\urcorner+1)}.$$
By (8), we have $V=V_{z\upharpoonright(\ulcorner i_0,j\urcorner+1)}$ and $G=G_{z\upharpoonright(\ulcorner i_0,j\urcorner+1)}$ for all $j\ge J_0$. If $j>J_0$ with $z(\ulcorner i_0,j\urcorner)=1$, by (8)(b), we have $\psi(z)\in W_n(G)$ for all $n\le\ulcorner i_0,j\urcorner$. Since there is infinitely many such $j$, we have $\psi(z)\in G\subseteq f^{-1}(V)$. Therefore, $f(\psi(z))\in V\subseteq O$.

If $z\notin\Omega$, we show that $f(\psi(z))\notin O$. If not, there exits $t_1$ and $m_1$ such that $\psi(z)\in G^{m_1}(V_{t_1})$. Fix an $i_1>\max\{m_1,{\rm lh}(t_1)\}$. Since $z\notin\Omega$, for large enough $j$, we have $z(\ulcorner i,j\urcorner)=0$ for all $i<i_1$, so $z\upharpoonright(\ulcorner i_1,j\urcorner+1)$ is an inheritor. Let $J_1$ be the largest $j$ such that $z\upharpoonright(\ulcorner i_1,j\urcorner+1)$ is an innovator. Denote
$$F=F_{z\upharpoonright(\ulcorner i_1,J_1\urcorner+1)}.$$
By (8), we have $F=F_{z\upharpoonright(\ulcorner i_1,j\urcorner+1)}$ for all $j\ge J_1$. From (5) and (6), we see that $F\cap U_{z\upharpoonright(\ulcorner i_1,j\urcorner+1))}\ne\emptyset$ for any $j>J_1$, and hence $\psi(z)\in F$. It follows from (9) that $F\cap U_{z\upharpoonright(\ulcorner i_1,J_1\urcorner+1)}\cap G^{m_1}(V_{t_1})=\emptyset$. Thus $\psi(z)\notin G^{m_1}(V_{t_1})$. A contradiction!

Now we turn to the construction.

First, set $D,P,h,\mathcal B_P,\mathcal G_V$ as follows:
\begin{enumerate}
\item[(i)]$P=Y$, $D=X$, and $h=f$;
\item[(ii)] $\mathcal B_P$ is a countable basis of $Y$;
\item[(iii)] for each $V\in\mathcal B_P$, let $\mathcal G_V=\{G^m(V):m<\omega\}$.
\end{enumerate}
Applying Lemma~\ref{PDtoVAF} with these $D,P,h,\mathcal B_P,\mathcal G_V$, we get an open set $V\subseteq Y$, a $G_\delta$ set $G\in\mathcal G_V$, and a non-empty closed set $F\subseteq X$. Then put
$$V_\emptyset=V,\quad G_\emptyset=G,\quad F_\emptyset=F,\quad U_\emptyset=X.$$

Secondly, assume that we have constructed $V_t,G_t,F_t,U_t,$ and $U_t^w$ for $t,w\prec s^\smallfrown 0$. We will define for $s^\smallfrown 0$ and $s^\smallfrown 1$. We consider the following two cases:

{\sl Case 1.} Assume $s^\smallfrown 0$ is an inheritor. Let $v=v(s^\smallfrown i),u=u(s^\smallfrown i)$. For $i=0,1$, put
$$V_{s^\smallfrown i}=V_v,\quad G_{s^\smallfrown i}=G_v,\quad F_{s^\smallfrown i}=F_v.$$
Note that $s=u$ or $s$ is also an inheritor with $u(s)=u$. By (5) and (8)(a), $U_s^{s_M}\cap F_u\ne\emptyset$.

{\sl Subcase 1.1.} If ${\rm col}(s^\smallfrown 0^\smallfrown 0)>0$, then $s^\smallfrown 0^\smallfrown 0$ is still an inheritor. We can find a $U_{s^\smallfrown 0}$ such that
$$\overline{U_{s^\smallfrown 0}}\subseteq U_s^{s_M},\quad U_{s^\smallfrown 0}\cap F_u\ne\emptyset,\quad{\rm diam}(\overline{U_{s^\smallfrown 0}})\le 2^{-({\rm lh}(s)+1)}.$$
By (10)(b), we see $\mathcal F_u$ is $P_{s_M}$-sharp in $U_s^{s_M}$. By Proposition~\ref{sharp2}, $\mathcal F_v$ is $P_{s_M}$-sharp in some open set $U\subseteq U_s^{s_M}$, and hence $U\cap F_v\ne\emptyset$. Denote $W=\bigcap_{n\le{\rm lh}(s)}W_n(G_v)$. Since $W\supseteq G_v$, by (7) we have $W\cap F_v$ is open dense in $F_v$, and hence $W\cap U\cap F_v\ne\emptyset$. We can find $U_{s^\smallfrown 1}$ such that
$$\overline{U_{s^\smallfrown 1}}\subseteq U\cap W,\quad U_{s^\smallfrown 1}\cap F_v\ne\emptyset,\quad{\rm diam}(\overline{U_{s^\smallfrown 1}})\le 2^{-({\rm lh}(s)+1)}.$$
By shrinking we may assume that $U_{s^\smallfrown 0}\cap U_{s^\smallfrown 1}=\emptyset$. For $i=0,1$, we set $U_{s^\smallfrown 0}^{s^\smallfrown i}=U_{s^\smallfrown 0}$, $U_{s^\smallfrown 1}^{s^\smallfrown 1}=U_{s^\smallfrown 1}$, and for other $t$, set $U_t^{s^\smallfrown i}=U_t^{s_M}$.

{\sl Subcase 1.2.} If ${\rm col}(s^\smallfrown 0^\smallfrown 0)=0$, then $s^\smallfrown 0^\smallfrown 0$ is an innovator. We define $U_{s^\smallfrown i}$ for $i=0,1$ similar to $U_{s^\smallfrown 1}$ in Subcase 1.1 with one more requirement $U_{s^\smallfrown 0}\cap U_{s^\smallfrown 1}=\emptyset$.

It is trivial to check clauses (0)--(9). Note that $P_{s^\smallfrown 0}=P_{s^\smallfrown 1}=P_{s_M}$ and $\mathcal F_{s^\smallfrown 0}=\mathcal F_{s^\smallfrown 1}=\mathcal F_v$. Since (10) holds for $s_M$, it also holds for $s^\smallfrown 0$ and $s^\smallfrown 1$.

{\sl Case 2.} Assume $s^\smallfrown 0$ is an innovator. We inductively define $V^l,G^l,F^l$ and $U^l$ for each $l<\omega$ as the following:

Since $s\preceq s_M\prec s^\smallfrown 0$, by (10)(a), we have $\mathcal F_s$ is $P_{s_M}$-sharp in $U_s^{s_M}$. So
$$f\upharpoonright(f^{-1}(P_{s_M})\cap F_s\cap U_s^{s_M})\notin\d23.$$
Denote $F^{-1}=F_s,U^{-1}=U_s^{s_M}$. Assume that we have defined $V^k,G^k,F^k$ and $p^k$ for $k<l$. Set $D,P,h,\mathcal B_P,\mathcal G_V$ as follows:
\begin{enumerate}
\item[(i)] $P=P_{s_M}\setminus\bigcup_{k<l}\overline{V^k}$, $D=f^{-1}(P)\cap F^{l-1}\cap U^{l-1}$, and $h=f\upharpoonright D$;
\item[(ii)] $\mathcal B_P$ is a countable basis of $P$ such that $\overline{V}\subseteq P$ for each $V\in\mathcal B_P$;
\item[(iii)] for each $V\in\mathcal B_P$, let $\mathcal G_V=\{D\cap G^m(V):m<\omega\}$.
\end{enumerate}
Applying Lemma~\ref{PDtoVAF} with these $D,P,h,\mathcal B_P,\mathcal G_V$, we get an open set $V\subseteq Y$, a $G_\delta$ set $G=G^m(V)$ for some $m<\omega$, and a closed set $F\subseteq\overline D\subseteq F^{l-1}\subseteq F_s$ with $F\cap U^{l-1}\supseteq F\cap D\ne\emptyset$. Denote $V^l=V, G^l=G, F^l=F$. If $\mathcal F_s^\smallfrown F^l$ is $P_{s_M}\setminus\overline{V^l}$-sharp in $U^{l-1}$, set $U^l=U^{l-1}$. Otherwise, since Lemma~\ref{PDtoVAF} implies that $F^l$ is $P_{s_M}\setminus\overline{V^l}$-sharp in $U^{l-1}$, we can find an open set $U^l\subseteq U^{l-1}$ with $U^l\cap F^l\ne\emptyset$ such that $\mathcal F_s$ is not $P_{s_M}\setminus\overline{V^l}$-sharp in any open set $U\subseteq U^l$. This complete the induction.

For $s\prec t\preceq s_M$, if $t^\smallfrown 0$ is an innovator, it follows from (10)(a) that $\mathcal F_t$ is $P_{s_M}$-sharp in $U_t^{s_M}$. By Lemma~\ref{sharp}, we can find a natural number $L_t$ such that, for any $l\ge L_t$, $\mathcal F_t$ is $P_{s_M}\setminus\overline{V^l}$-sharp in some open set $U_t^l\subseteq U_t^{s_M}$. If $t^\smallfrown 0$ is an inheritor, from (10)(b) and Lemma~\ref{sharp}, we can also find a natural number $L_t$ such that, for any $l\ge L_t$, $\mathcal F_{u(t^\smallfrown 0)}$ is $P_{s_M}\setminus\overline{V^l}$-sharp in some open set $U_t^l\subseteq U_t^{s_M}$. Moreover, assume for contradiction that there exist $l_0<\cdots<l_m$ with $m={\rm lh}(\mathcal F_s)$ such that $\mathcal F_s$ is not $P_{s_M}\setminus\overline{V^{l_j}}$-sharp in any open set $U\subseteq U^{l_j}$ for $j\le m$. This contradicts Lemma~\ref{sharp}, because $U^{l_m}\subseteq U_s^{s_M}$ and $U^{l_m}\cap F^{l_m}\ne\emptyset$ implies that $\mathcal F_s$ is $P_{s_M}$-sharp in $U^{l_m}$. Therefore, comparing with the definition of $U^l$, we can find an natural number $L'$ such that, for any $l\ge L'$, $\mathcal F_s^\smallfrown F^l$ is $P_{s_M}\setminus\overline{V^l}$-sharp in $U^l$. Then we set
$$L=\max\{L',L_t:s\prec t\preceq s_M\}$$
and $U_t^{s^\smallfrown 0}=U_t^{s^\smallfrown 1}=U_t^L$ for $t\prec s^\smallfrown 0\prec t^\smallfrown 0$, i.e., for $s\prec t\preceq s_M$.

In the end, denote
$$A=\bigcup_{{\rm lh}(t)\le{\rm lh}(s)}\bigcup_{m\le{\rm lh}(s)}G^m(V_t).$$
Then $A$ is $G_\delta$ set. By (3) and $V^L\subseteq P_{s_M}$, we have $G^L\cap A=\emptyset$. It follows from Lemma~\ref{PDtoVAF} that $G^L\cap F^L$ is dense in $F^L$, so $A\cap F^L$ is nowhere dense in $F^L$. We can find two open sets $U_{s^\smallfrown 0}$ and $U_{s^\smallfrown 1}$ such that $U_{s^\smallfrown 0}\cap U_{s^\smallfrown 1}=\emptyset$, and for $i=0,1$, we have
$$\overline{U_{s^\smallfrown i}}\subseteq U^L,\quad U_{s^\smallfrown i}\cap F^L\ne\emptyset,\quad U_{s^\smallfrown i}\cap F^L\cap A=\emptyset,\quad{\rm diam}(\overline{U_{s^\smallfrown i}})\le 2^{-({\rm lh}(s)+1)}.$$
Now put
$$V_{s^\smallfrown i}=V^L,\quad G_{s^\smallfrown i}=G^L,\quad F_{s^\smallfrown i}=F^L,$$
and $U_{s^\smallfrown 0}^{s^\smallfrown i}=U_{s^\smallfrown 0},U_{s^\smallfrown 1}^{s^\smallfrown 1}=U_{s^\smallfrown 1}$.
\end{proof}

\begin{corollary}\label{cor1}
Let $X$ be a Polish space, $Y$ a separable metrizable space, and let $f:X\to Y$. If $f\notin\d23$, then there exists a Cantor set $C\subseteq X$ such that $f\upharpoonright C\notin\d23$.
\end{corollary}

\begin{proof}
Let $\psi$ be the continuous embedding defined in Theorem~\ref{3to2}. Put $C=\psi(2^\omega)$.
\end{proof}

\section{The decomposability conjecture for $m=n=3$}

Before proving Theorem \ref{main} for $m=n=3$, we prove a known result first: for functions of Baire class $1$,
$$f^{-1}{\bf\Sigma}^0_3\subseteq{\bf\Sigma}^0_3\Rightarrow f\in\d13.$$
This is an easy corollary of Solecki's theorem (see \cite[Theorem 4.1]{solecki1}), since $f\in\d13\iff f\in{\rm dec}({\bf\Sigma}^0_1)$ and $P^{-1}{\bf\Sigma}^0_3\not\subseteq{\bf\Sigma}^0_3$. Furthermore, this result is also a special case of \cite[Corollary 1.2]{PS}, \cite[Corollary 5.11]{motto_ros}, or \cite[Theorem 1.1]{GKN}. In order to show a completely different method of proof, we present a direct proof which follows the same idea as in the previous section. The readers can skip directly to Theorem~\ref{3to1}.

\begin{lemma}\label{PDtoVFE}
Let $X,P$ be two separable metrizable spaces, and let $D\subseteq X$, and $h:D\to P$ a function of Baire class $1$. Let $\mathcal B_P$ be a countable topological basis of $P$. If $h\notin\d13$, then there exist a $V\in\mathcal B_P$ and two closed sets $E\subseteq F\subseteq\overline D$ satisfying:
\begin{enumerate}
\item[(a)] for any open set $U$ with $E\cap U\ne\emptyset$, we have
$$h\upharpoonright(h^{-1}(V)\cap U\cap E)\notin\d13;$$
\item[(b)] for any open set $U$ with $F\cap U\ne\emptyset$, we have
$$h\upharpoonright(h^{-1}(\overline{V}^c)\cap F\cap U)\notin\d13;$$
\item[(c)] $E\cap D\ne\emptyset$.
\end{enumerate}
\end{lemma}

\begin{proof}
Let $\{U_k:k<\omega\}$ be a topological basis of $X$. For any $B\in\mathcal B_P$, we denote
$$F^B=\{x\in X:\forall k(x\in U_k\Rightarrow h\upharpoonright(h^{-1}(B^c)\cap U_k)\notin\d13)\}.$$
It is trivial to see that
\begin{enumerate}
\item[(i)] $F^B$ is closed,
\item[(ii)] $h\upharpoonright(h^{-1}(B^c)\setminus F^B)\in\d13$, and
\item[(iii)] for any open set $U$ with $F^B\cap U\ne\emptyset$,
$$h\upharpoonright(h^{-1}(B^c)\cap F^B\cap U)\notin\d13.$$
\end{enumerate}

Assume for contradiction that, $h\upharpoonright(h^{-1}(B)\cap F^B)\in\d13$ for any $B\in\mathcal B_P$. We denote
$$H_1=\bigcup_{B\in\mathcal B_P}(h^{-1}(B)\cap F^B),$$
$$H_2=\bigcup_{B\in\mathcal B_P}(h^{-1}(B^c)\setminus F^B).$$
It is straightforward to check that, $h\upharpoonright H_i\in\d13$ for $i=1,2$.

Denote $H_3=D\setminus(H_1\cup H_2)$. For any $x\in H_3$ and any $B\in\mathcal B_P$, we have
$$h(x)\in B \Rightarrow x\in h^{-1}(B) \Rightarrow x\notin F^B,$$
$$h(x)\notin B\Rightarrow x\in h^{-1}(B^c) \Rightarrow x\in F^B.$$
So $h\upharpoonright H_3$ is continuous.

Let $\tilde Y\supseteq Y$ be a Polish space. By Kuratowski's theorem (cf. \cite[(3.8)]{kechris}), there is a $G_\delta$ set $G\supseteq H_3$ and a continuous function $g:G\to\tilde Y$ such that $g\upharpoonright H_3=h\upharpoonright H_3$. Put $H=\{x\in D\cap G:h(x)=g(x)\}$. Since $h$ is of Baire class $1$, we see $H$ is $G_\delta$ subset of $D$ and $H_1\cup H_2\cup H=D$. Note that
$H_1$ is $F_\sigma$ subset of $D$ and $H_2$ is ${\bf\Sigma}^0_3$ subset of $D$. It follows that $h\in\d13$. A contradiction!

Therefore, there exists a $B\in\mathcal B_P$ such that
$$h\upharpoonright(h^{-1}(B)\cap F^B)\notin\d13.$$
Since $B=\bigcup\{V\in\mathcal B_P:\overline{V}\subseteq B\}$,
we can find a $V\in\mathcal B_P$ with $\overline{V}\subseteq B$ such that
$$h\upharpoonright(h^{-1}(V)\cap F^B)\notin\d13.$$

In the end, define
$$E=\{x\in F^B:\forall k(x\in U_k\Rightarrow h\upharpoonright(h^{-1}(V)\cap U_k)\notin\d13)\}.$$
Then we have $E\cap D\ne\emptyset$, and
$$h\upharpoonright(h^{-1}(V)\cap U\cap E)\notin\d13$$
for any open set with $U\cap E\ne\emptyset$. Note that $\overline{V}^c\supseteq B^c$. So $V,E$ and $F^B$ satisfy clauses (a)--(c) as desired.
\end{proof}

In the rest of this section, we fix $X$ be a Polish space, $Y$ a separable metrizable space, and $f:X\to Y$ a ${\bf\Sigma}^0_3$-measurable function.

\begin{definition}
Let $\mathcal F=\langle F_0,\cdots,F_k\rangle$ be a finite sequence of closed sets of $X$ with $F_0\supseteq\cdots\supseteq F_k$, $U$ an open subset of $X$, and let $\mathcal P=\langle P_0,\cdots,P_k\rangle$ be a sequence of pairwise disjoint subsets of $Y$.
\begin{enumerate}
\item[(i)] If $k=0$, i.e., $\mathcal F=\langle F_0\rangle,\mathcal P=\langle P_0\rangle$, then we say $\mathcal F$ is $\mathcal P$-{\bf sharp} in $U$ if $U\cap F_0\ne\emptyset$, and for any open set $U'\subseteq U$ with $U'\cap F_0\ne\emptyset$, we have
$$f\upharpoonright(f^{-1}(P_0)\cap F_0\cap U')\notin\d13.$$
We also say $F_0$ itself is $P_0$-sharp in $U$ for brevity.
\item[(ii)] If $k>0$, then we say $\mathcal F$ is $\mathcal P$-{\bf sharp} in $U$ if $F_k$ is $P_k$-sharp in $U$, and for any open set $U'\subseteq U$ with $U'\cap F_k$, $\mathcal F\upharpoonright k$ is $\mathcal P\upharpoonright k$-sharp in some open set $U''\subseteq U'$.
\end{enumerate}
\end{definition}

\begin{proposition}\label{shp1}
Suppose $\mathcal F=\langle F_0,\cdots,F_k\rangle$ is $\mathcal P$-sharp in $U$. Then for any $U'\subseteq U$ with $U'\cap F_k\ne\emptyset$, we have $\mathcal F$ is $\mathcal P$-sharp in $U'$.
\end{proposition}

\begin{proposition}\label{shp2}
Suppose $\mathcal F$ is $\mathcal P$-sharp in $U$. Then for any $m<{\rm lh}(\mathcal F)$, $\mathcal F\upharpoonright m$ is $\mathcal P\upharpoonright m$-sharp in some open set $U'\subseteq U$.
\end{proposition}

Let $\mathcal P=\langle P_0,\cdots,P_k\rangle$, $0\le j\le l$, and let $C\subseteq P_j$. We denote
$$\mathcal P\setminus C=\langle P_0,\cdots,P_j\setminus C,\cdots,P_k\rangle.$$

\begin{lemma}\label{shp}
Let $\mathcal F=\langle F_0,\cdots,F_k\rangle,\mathcal P=\langle P_0,\cdots,P_k\rangle$. Suppose $\mathcal F$ is $P$-sharp in $U$. Let $0\le j\le k$ and $(C_l)_{l<m}$ be a sequence of pairwise disjoint closed subsets of $P_j$. Then there exist at most one $l$ such that $\mathcal F$ is not $\mathcal P\setminus C_l$-sharp in any open set $U'\subseteq U$.
\end{lemma}

\begin{proof}
We begin with $k=j=0$. Without loss of generality, suppose there exists an $l<m$, say, $l=0$, such that $F_0$ is not $P_0\setminus C_0$-sharp in $U$. Then there exists an open set $U_0\subseteq U$ with $U_0\cap F_0\ne\emptyset$ such that
$$f\upharpoonright(f^{-1}(P_0\setminus C_0)\cap F_0\cap U_0)\in\d13.$$
Assume for contradiction that there exists $l\ne 0$ such that $F_0$ is not $P_0\setminus C_l$-sharp in $U_0$, then there is an open set $U_l\subseteq U_0$ with $U_l\cap F_0\ne\emptyset$ such that
$$f\upharpoonright(f^{-1}(P_0\setminus C_l)\cap F_0\cap U_l)\in\d13.$$
Since $C_0$ and $C_l$ are disjoint closed subsets of $P_0$, Proposition~\ref{proposition} gives
$$f\upharpoonright(f^{-1}(P_0)\cap F_0\cap U_l)\in\d13,$$
contradicting that $F_0$ is $P_0$-sharp in $U$.

For $k>0$, assume that we have proved for all $k'<k$.

{\sl Case 1.} If $j=k$, since $F_k$ is $P_k$-sharp in $U$, from the arguments for $k=0$ above, we may assume that there is an open set $U_0\subseteq U$ with $U_0\cap F_k\ne\emptyset$ such that $F_k$ is $P_k\setminus C_l$-sharp in $U_0$ for any $l\ne 0$. It follows that $\mathcal F$ is $\mathcal P\setminus C_l$-sharp $U_0$ for any $l\ne 0$.

{\sl Case 2.} If $j<k$, assume for contradiction that there are more than one $l$, say, $l=0,1$, such that $\mathcal F$ is not $\mathcal P\setminus C_l$-sharp in any open set $U'\subseteq U$. Particularly, $\mathcal F$ is not $\mathcal P\setminus C_0$-sharp in $U$. Note that $F_k$ is $P_k\setminus C_l$-sharp in $U$ for any $l<m$, so there exists an $U_0\subseteq U$ with $U_0\cap F_k\ne\emptyset$ such that $\mathcal F\upharpoonright k$ is not $(\mathcal P\upharpoonright k)\setminus C_0$-sharp in any open set $U'\subseteq U_0$. Similarly, we can find an open set $U_1\subseteq U_0$ with $U_1\cap F_k\ne\emptyset$ such that $\mathcal F\upharpoonright k$ is not $(\mathcal P\upharpoonright k)\setminus C_1$-sharp in any $U'\subseteq U_1$. By Propositions~\ref{shp1} and~\ref{shp2}, there is an open set $U^*\subseteq U_1$ such that $\mathcal F\upharpoonright k$ is $\mathcal P$-sharp in $U^*$, contradicting the induction hypothesis.
\end{proof}

\begin{theorem}\label{3to1part}
Let $X$ be a Polish space, $Y$ a separable metrizable space, and let $f:X\to Y$ be of Baire class $1$. If $f^{-1}{\bf\Sigma}^0_3\subseteq{\bf\Sigma}^0_3$, then $f\in\d13$.
\end{theorem}

\begin{proof}
Assume for contradiction that $f\notin\d13$. We will define a continuous embedding $\psi:2^\omega\to X$ and an $G_\delta$ set $G\subseteq Y$ such that $\psi^{-1}(f^{-1}(Y\setminus G))=\Omega$. Thus $f^{-1}(G)$ is ${\bf\Pi}^0_3$-complete subset of $X$, contradicting $f^{-1}{\bf\Sigma}^0_3\subseteq{\bf\Sigma}^0_3$.

It it well known that $Y$ is homeomorphic to a subspace of $\R^\omega$. Without loss of generality, we may assume $Y=\R^\omega$. Granting this assumption, we can fix a sequence of continuous functions $f_n:X\to Y$ pointwisely converging to $f$. Fix a compatible metric $d$ on $X$ with $d\le 1$.

For $s\ne\emptyset$, let ${\rm lh}(s)=\ulcorner i,j\urcorner+1$. Now we {\bf redefine} inheritors and innovators. We say $s$ is an {\bf inheritor} if $j> 0$ and $s(\ulcorner k,i+j-k\urcorner)=0$ for any $k\le i$ (note: it was for any $k<i$ in the definition of inheritor in Theorem~\ref{3to2}), otherwise we say $s$ is an {\bf innovator}. Note that $s$ is always an innovator if $j=0$ or $s(\ulcorner i,j\urcorner)=1$.

We will inductively construct for each $s\in 2^{<\omega}$ an open set $V_s$ of $Y$, two closed sets $E_s,F_s$ of $X$, an open set $U_s$ of $X$, and a sequence of open sets $(U_s^w)_{s\preceq w\prec s^\smallfrown 0}$ of $X$ satisfying the following:
\begin{enumerate}
\item[(0)] ${\rm diam}(\overline{U_s})\le 2^{-{\rm lh}(s)}$, $U_{s^\smallfrown 0}\cap U_{s^\smallfrown 1}=\emptyset$, $\overline{U_{s^\smallfrown 0}},\overline{U_{s^\smallfrown 1}}\subseteq U_s^w$;
\item[(1)] $F_{s^\smallfrown 1}\subseteq F_{s^\smallfrown 0}$;
\item[(2)] $\overline{V}_s\subseteq V_\emptyset$ and $F_s\subseteq E_\emptyset$ for any $s\ne\emptyset$;
\item[(3)] for any $s,t\ne\emptyset$ with ${\rm row}(s)={\rm row}(t)$, we have $V_s=V_t$ or $\overline{V_s}\cap\overline{V_t}=\emptyset$;
\item[(4)] if ${\rm col}(s)>0$, then $\overline{V_{s^\smallfrown 0}},\overline{V_{s^\smallfrown 1}}\subseteq V_s$ and $F_{s^\smallfrown 0}\subseteq F_s$;
\item[(5)] $E_s\subseteq F_s$;
\item[(6)] $E_s\cap U_s^w\ne\emptyset$ for each $w$;
\item[(7)] $U_s=U_s^s$, and $U_s^{w_1}\supseteq U_s^{w_2}$ for $w_1\preceq w_2$;
\item[(8)] if $s$ is an inheritor, then we have
$$V_s=V_{v(s)},\quad F_s=F_{v(s)},\quad E_s=E_{u(s)};$$
\item[(9)] if $s$ is an innovator, then $\overline{V_s}\cap\overline{V_t}=\emptyset$ for any $t$ with $t\prec s$ and ${\rm row}(t)={\rm row}(s)$; furthermore, there exists $n\ge{\rm lh}(s)$ such that $f_n(U_s)\subseteq V_s$;
\item[(10)] if $s\ne\emptyset$, by letting $V_s^-=\left\{\begin{array}{ll}V_{s\upharpoonright({\rm lh}(s)-1)}, &{\rm row}(s)>0,\cr V_{\emptyset}, &{\rm row}(s)=0,\end{array}\right.$
$$P_s^r=V_s^-\setminus\bigcup\{\overline{V_t}:t\preceq r,\;{\rm row}(t)={\rm row}(s)\},$$
$$\mathcal P_s^r=\langle P_{s\upharpoonright({\rm lh}(s)-{\rm row}(s))}^r,\cdots,P_s^r,V_s\rangle,$$
$$\mathcal F_s=\langle F_{s\upharpoonright({\rm lh}(s)-{\rm row}(s))},\cdots,F_s,E_s\rangle,$$
then for any $t\preceq s\prec t^\smallfrown 0$, we have
\\ (a) if $t^\smallfrown 0$ is an innovator, then $\mathcal F_t$ is $\mathcal P_t^s$-sharp in $U_t^s$;
\\ (b) if $t^\smallfrown 0$ is an inheritor, then $\mathcal F_{u(t^\smallfrown 0)}$ is $\mathcal P_{u(t^\smallfrown 0)}^s$-sharp in $U_t^s$.
\end{enumerate}

When we complete the construction, for any $z\in 2^\omega$, we set $\psi(z)$ to be the unique element of $\bigcap_kU_{z\upharpoonright k}$. From (0) and (7), $\psi$ is continuous embedding from $2^\omega$ to $X$. Put
$$G_m=\bigcup_{{\rm row}(t)=m}V_t,\quad G=\bigcap_{m<\omega}G_m.$$

If $z\in\Omega$, there exist $i_0<\omega$ and a strictly increasing sequence $j_k>0$ with $z(\ulcorner i_0,j_k\urcorner)=1$ for any $k<\omega$. Since $z\upharpoonright(\ulcorner i_0,j_k\urcorner+1)$ is an innovator, by (9), there is $n_k>\ulcorner i_0,j_k\urcorner$ such that $f_{n_k}(\psi(z))\in V_{z\upharpoonright(\ulcorner i_0,j_k\urcorner+1)}$. It follows from (9) that $f(\psi(z))\notin V_t$ whenever ${\rm row}(t)=i_0$. Thus
$$f(\psi(z))\notin G_{i_0}\supseteq G.$$

If $z\notin\Omega$, we show that $f(\psi(z))\in G$. For any $m<\omega$, there exists $J_m<\omega$ such that $z(\ulcorner i,j\urcorner)=0$ for any $i\le m$ and any $j>J_m$. So $z\upharpoonright(\ulcorner m,j\urcorner+1)$ is an inheritor for any $j>J_m$. Denote
$$V_m=V_{z\upharpoonright(\ulcorner m,J_m\urcorner+1)},\quad u^m_j=u(z\upharpoonright(\ulcorner m,j\urcorner+1)).$$
By (8), we have $V_m=V_{z\upharpoonright(\ulcorner m,j\urcorner+1)}$ for all $j>J_m$. Since all $u^m_j$ are innovators, by (9) we can find an $n_j\ge{\rm lh}(u^m_j)$ such that $f_{n_j}(\psi(z))\in V_{u^m_j}$. By (4) and (8) we have $f_{n_j}(\psi(z))\in V_m$ for all $j>J_m$. So $f(\psi(z))\in\overline{V_m}$ for each $m$. Again by (4) we have $\overline{V_{m+1}}\subseteq V_m$ for any $m<\omega$. So
$$f(\psi(z))\in\overline{V_{m+1}}\subseteq V_m\subseteq G_m$$
for all $m<\omega$. It follows that $f(\psi(z))\in G$.

Now we turn to the construction.

First, set $D,P,h,\mathcal B_P$ as follows:
\begin{enumerate}
\item[(i)] $P=Y,D=X,h=f$;
\item[(ii)] $\mathcal B_P$ is a countable basis of $Y$.
\end{enumerate}
Applying Lemma~\ref{PDtoVFE} with these $D,P,h,\mathcal B_P$, we get an open set $V$ of $Y$ and two closed sets $E\subseteq F$ of $X$. Then put
$$V_\emptyset=V,\quad F_\emptyset=F,\quad E_\emptyset=E,\quad U_\emptyset=X.$$

Secondly, assume that we have constructed $V_t,E_t,F_t,U_t,$ and $U_t^w$ for $t,w\prec s^\smallfrown 0$. We will define for $s^\smallfrown 0$ and $s^\smallfrown 1$. We consider the following two cases:

{\sl Case 1.} Assume $s^\smallfrown 0$ is an inheritor. Let $v=v(s^\smallfrown 0),u=u(s^\smallfrown 0)$. Put
$$V_{s^\smallfrown 0}=V_v,\quad F_{s^\smallfrown 0}=F_v,\quad E_{s^\smallfrown 0}=E_u.$$
Note that either $s=u$, or $s$ is also an inheritor with $u(s)=u$, so $E_s=E_u$. By (7), $E_u\cap U_s^{s_M}\ne\emptyset$, so we can define an open set $U_{s^\smallfrown 0}$ such that
$$\overline{U_{s^\smallfrown 0}}\subseteq U_s^{s_M},\quad U_{s^\smallfrown 0}\cap E_u\ne\emptyset,\quad{\rm diam}(\overline{U_{s^\smallfrown 0}})\le 2^{-({\rm lh}(s)+1)}.$$
We set $U_{s^\smallfrown 0}^{s^\smallfrown 0}=U_{s^\smallfrown 0}$ and $U_t^{s^\smallfrown 0}=U_t^{s_M}$ for other $t$.

To check (0)--(10), the only nontrivial one is (10)(a) with $t=s^\smallfrown 0$. Note that, if $s^\smallfrown 0^\smallfrown 0$ is innovator, then ${\rm col}(s^\smallfrown 0^\smallfrown 0)=0$, i.e., $u=v$, so $\mathcal P_{s^\smallfrown 0}^{s^\smallfrown 0}=\mathcal P_u^{s_M}$ and $\mathcal F_{s^\smallfrown 0}=\mathcal F_u$. Since (10) holds for $s_M$, it holds for $s^\smallfrown 0$ too.

By shrinking, we may assume $E_u\cap(U_s^{s_M}\setminus\overline{U_{s^\smallfrown 0}})\ne\emptyset$. By (10)(b) and Proposition~\ref{shp1}, we see $\mathcal F_u$ is $\mathcal P_u^{s_M}$-sharp in $U_s^{s_M}\setminus\overline{U_{s^\smallfrown 0}}$. By Proposition~\ref{shp2}, $\mathcal F_u\upharpoonright({\rm row}(v)+1)$ is $\mathcal P_u^{s_M}\upharpoonright({\rm row}(v)+1)$-sharp in some open set $U\subseteq(U_s^{s_M}\setminus\overline{U_{s^\smallfrown 0}})$. Thus $F_v$ is $P_v^{s_M}$-sharp in $U$, and hence
$$f\upharpoonright(f^{-1}(P_s^{s_M})\cap F_v\cap U)\notin\d13.$$

We inductively define $V^l,E^l$, and $F^l$ for each $l<\omega$. Denote $F^{-1}=F_v$. Assume that we have defined $V^k,E^k$, and $F^k$ for $k<l$. Set $D,P,h,\mathcal B_P$ as follows:
\begin{enumerate}
\item[(i)] $P=P_s^{s_M}\setminus\bigcup_{k<l}\overline{V^k}$, $D=F^{l-1}\cap U\cap f^{-1}(P)$, $h=f\upharpoonright D$;
\item[(ii)] $\mathcal B_P$ is a countable basis of $P$ such that $\overline{V}\subseteq P$ for each $V\in\mathcal B_P$.
\end{enumerate}
Applying Lemma~\ref{PDtoVFE} with these $D,P,h,\mathcal B_P$, we get an open set $V$ of $Y$ and two closed sets $E\subseteq F\subseteq\overline D\subseteq F^{l-1}$ with $E\cap U\supseteq E\cap D\ne\emptyset$. Denote $V^l=V, E^l=E$, and $F^l=F$. This complete the induction.

For $s\prec t\preceq s^\smallfrown 0$, if $t^\smallfrown 0$ is an innovator, it follows from (10)(a) that $\mathcal F_t$ is $\mathcal P_t^{s^\smallfrown 0}$-sharp in $U_t^{s^\smallfrown 0}$. By Lemma~\ref{shp}, we can find an natural number $L_t$ such that, for any $l\ge L_t$, $\mathcal F_t$ is $\mathcal P_t^{s^\smallfrown 0}\setminus\overline{V^l}$-sharp in some $U_t^l\subseteq U_t^{s^\smallfrown 0}$. If $t^\smallfrown 0$ is an inheritor, from (10)(b) and Lemma~\ref{shp}, we can also find an natural number $L_t$ such that, for any $l\ge L_t$, $\mathcal F_{u(t^\smallfrown 0)}$ is $\mathcal P_t^{s^\smallfrown 0}\setminus\overline{V^l}$-sharp in some open set $U_t^l\subseteq U_t^{s_M}$. Then we set
$$L=\max\{L_t:s\prec t\preceq s^\smallfrown 0\}$$
and $U_t^{s^\smallfrown 1}=U_t^L$ for $t\preceq s^\smallfrown 0\prec t^\smallfrown 0$, i.e., for $s\prec t\preceq s^\smallfrown 0$.

From Lemma~\ref{PDtoVFE} and $F^L\subseteq E_s$, we can see that $(\mathcal F_s\upharpoonright{\rm row}(s^\smallfrown 1))^\smallfrown F^L{}^\smallfrown E^L$ is $(\mathcal P_s^{s_M}\setminus\overline{V^L})^\smallfrown V^L$-sharp in $U$.

Pick an $x\in(f^{-1}(V^L)\cap E^L\cap U)$. Since $f(x)\in V^L$, there is an $n>{\rm lh}(s)$ such that $f_n(x)\in V^L$. Then we can define an open set $U_{s^\smallfrown 1}$ such that
$$\overline{U_{s^\smallfrown 1}}\subseteq U,\quad f_n(U_{s^\smallfrown 1})\subseteq V^L,\quad x\in U_{s^\smallfrown 1},\quad{\rm diam}(\overline{U_{s^\smallfrown 1}})\le 2^{-({\rm lh}(s)+1)}.$$
Then put
$$V_{s^\smallfrown 1}=V^L,\quad E_{s^\smallfrown 1}=E^L,\quad F_{s^\smallfrown 1}=F^L,$$
and $U_{s^\smallfrown 1}^{s^\smallfrown 1}=U_{s^\smallfrown 1}$.

{\sl Case 2.} Assume $s^\smallfrown 0$ is an innovator. Since $s\preceq s_M\prec s^\smallfrown 0$, by (10)(a), we have $\mathcal F_s$ is $\mathcal P_s^{s_M}$-sharp in $U_s^{s_M}$. Thus $E_s$ is $V_s$-sharp in $U_s^{s_M}$, and hence
$$f\upharpoonright(f^{-1}(V_s)\cap E_s\cap U_s^{s_M})\notin\d13.$$
Set $D,P,h,\mathcal B_P$ as follows:
\begin{enumerate}
\item[(i)] $P=V_s$, $D=E_s\cap U_s^{s_M}\cap f^{-1}(P)$, $h=f\upharpoonright D$;
\item[(ii)] $\mathcal B_P$ is a countable basis of $P$ such that $\overline{V}\subseteq P$ for each $V\in\mathcal B_P$.
\end{enumerate}
Applying Lemma~\ref{PDtoVFE} with these $D,P,h,\mathcal B_P$, we get an open set $V$ of $Y$ and two closed sets $E\subseteq F\subseteq\overline D\subseteq E_s$ with $E\cap U_s^{s_M}\supseteq E\cap D\ne\emptyset$. From Lemma~\ref{PDtoVFE} and $F\subseteq E_s$, we can see that $(\mathcal F_s\upharpoonright{\rm row}(s^\smallfrown 0))^\smallfrown F^\smallfrown E$ is $(\mathcal P_s^{s_M}\setminus\overline{V})^\smallfrown V$-sharp in $U_s^{s_M}$.

Pick an $x\in(f^{-1}(V)\cap E\cap U_s^{s_M})$. Since $f(x)\in V$, there is an $n>{\rm lh}(s)$ such that $f_n(x)\in V$. Then we can an open set $U_{s^\smallfrown 0}$ such that
$$\overline{U_{s^\smallfrown 0}}\subseteq U_s^{s_M},\quad f_n(U_{s^\smallfrown 0})\subseteq V,\quad x\in U_{s^\smallfrown 0},\quad{\rm diam}(\overline{U_{s^\smallfrown 0}})\le 2^{-({\rm lh}(s)+1)}.$$
Then put
$$V_{s^\smallfrown 0}=V,\quad E_{s^\smallfrown 0}=E,\quad F_{s^\smallfrown 0}=F,$$
and $U_{s^\smallfrown 0}^{s^\smallfrown 0}=U_{s^\smallfrown 0}$, and $U_t^{s^\smallfrown 0}=U_t^{s_M}$ for other $t$.

To check (0)--(10), it is trivial for $s=\emptyset$. For $s\ne\emptyset$, the only nontrivial clauses are (3), (9), and (10). Note that ${\rm row}(s^\smallfrown 0)>0$, so $V_{s^\smallfrown 0}^-=V_s$. Note also that either $s$ is also an innovator, or ${\rm col}(s^\smallfrown 0)=0$, i.e., $u(s)=v(s)$. In both cases, (4) and (9) imply that there is no $t\prec s^\smallfrown 0$ such that ${\rm row}(t)={\rm row}(s^\smallfrown 0)$ and $V_t^-=V_s$. So (3) and (9) hold. Therefore, $P_{s^\smallfrown 0}^{s^\smallfrown 0}=V_s\setminus\overline{V}$, thus $\mathcal P_{s^\smallfrown 0}^{s^\smallfrown 0}=(\mathcal P_s^{s_M}\setminus\overline{V})^\smallfrown V$. Similarly, $\mathcal P_t^{s^\smallfrown 0}=\mathcal P_t^{s_M}$ and $\mathcal P_{u(t^\smallfrown 0)}^{s^\smallfrown 0}=\mathcal P_{u(t^\smallfrown 0)}^{s_M}$ for $t\prec s^\smallfrown 0\prec t^\smallfrown 0$. Since (10) holds for $s_M$, it holds for $s^\smallfrown 0$ too.

By shrinking, we may assume $F\cap(U_s^{s_M}\setminus\overline{U_{s^\smallfrown 0}})\ne\emptyset$. By Lemma~\ref{PDtoVFE},
$$f\upharpoonright(f^{-1}(V_s\setminus\overline{V})\cap F\cap(U_s^{s_M}\setminus\overline{U_{s^\smallfrown 0}}))\notin\d13.$$
Now we define for $s^\smallfrown 1$ similar to the way in Case 1.
\end{proof}

\begin{theorem}\label{3to1}
Let $X$ be a Polish space, $Y$ a separable metrizable space, and let $f:X\to Y$. Then
$$f^{-1}{\bf\Sigma}^0_3\subseteq{\bf\Sigma}^0_3\iff f\in\d13.$$
\end{theorem}

\begin{proof}
The ``$\Leftarrow$'' part is trivial, we only prove the ``$\Rightarrow$'' part.

Since $f^{-1}{\bf\Sigma}^0_3\subseteq{\bf\Sigma}^0_3$ implies $f^{-1}{\bf\Sigma}^0_2\subseteq{\bf\Sigma}^0_3$, it follows from Theorem~\ref{3to2} that $f\in\d23$, i.e., there exists a sequence of $G_\delta$ set $X_n$ such that $\bigcup_nX_n=X$ and each $f\upharpoonright X_n$ is of Baire class 1. Then Theorem~\ref{3to1part} gives $f\upharpoonright X_n\in\d13$. Therefore, we have $f\in\d13$.
\end{proof}

\begin{corollary}
Let $X$ be a Polish space, $Y$ a separable metrizable space, and let $f:X\to Y$. If $f\notin\d13$, then there exists a Cantor set $C\subseteq X$ such that $f\upharpoonright C\notin\d13$.
\end{corollary}

\begin{proof}
If $f\notin\d23$, by Corollary~\ref{cor1}, there exists a Cantor set $C\subseteq X$ such that $f\upharpoonright C\notin\d23$. It is clear that $f\upharpoonright C\notin\d13$.

If $f\in\d23$, i.e., there exists a sequence of $G_\delta$ set $X_n$ such that $\bigcup_nX_n=X$ and each $f\upharpoonright X_n$ is of Baire class 1, then there is some $X_n$ such that $f\upharpoonright X_n\notin\d13$.  Let $\psi$ be the continuous embedding defined in Theorem~\ref{3to1part}. Put $C=\psi(2^\omega)$.
\end{proof}


\begin{thebibliography}{99}

\bibitem{andretta} A. Andretta, \textit{The SLO principle and the Wadge hierarchy}, in: Fundations of the formal sciences V, Studies in Logic, vol. 11, pp. 1--38. Coll. Publ., London, 2007.

\bibitem{GKN} V. Gregoriades, T. Kihara, and K. M. Ng, \textit{Turing degrees in Polish spaces and decomposability of Borel functions}, submitted, 2016.

\bibitem{jayne} J. E. Jayne, \textit{The space of class $\alpha$ Baire functions}, Bull. Amer. Math. Soc. 80 (1974), 1151--1156.

\bibitem{JR} J. E. Jayne and C. A. Rogers, \textit{Borel isomorphisms at the first level I}, Mathematika, 26 (1979), 125--156.

\bibitem{JR1} J. E. Jayne and C. A. Rogers, \textit{Fist level Borel functions and isomorphism}, J. Math. Pures Appl. 61 (1982), 177--205.

\bibitem{KMS} M. Ka\v cena, L. Motto Ros, and B. Semmes, \textit{Some observations on `A new proof of a theorem of Jayne and Rogers'}, Real Anal. Exchange 38 (2012/2013), 121--132.

\bibitem{kechris} A. S. Kechris, Classical Descriptive Set Theory, Graduate Texts in Mathematics, vol. 156, Springer-Verlag, 1995.

\bibitem{keldys} L. Keldy\v s, \textit{Sur les fonctions premi\' eres measurables B}, Soviet Math. Doklady 4 (1934), 192--197, in Russian and French.

\bibitem{kihara} T. Kihara, \textit{Decomposing Borel functions using the Shore-Slaman join theorem}, Fund. Math. 230 (2015), 1--13.

\bibitem{kuratowski} K. Kuratowski, \textit{Sur une g\' en\' eralisation de la notion d'hom\' eomorphie}, Fund. Math. 22 (1934), 206--220.

\bibitem{motto_ros} L. Motto Ros, \textit{On the structure of finite level and $\omega$-decomposable Borel functions}, J. Symbolic Logic 78 (2013), 1257--1287.

\bibitem{MS} L. Motto Ros and B. Semmes, \textit{A new proof of a theorem of Jayne and Rogers}, Real Anal. Exchange 35 (2009/2010), 195--204.

\bibitem{PS} J. Pawlikowski and M. Sabok, \textit{Decomposing Borel functions and structure at finite levels of the Baire hierarchy}, Ann. pure Appl. Logic 163 (2012), 1748--1764.

\bibitem{semmes} B. Semmes, A game for the Borel functions, Ph.D. thesis, ILLC, University of Amsterdam, Amsterdam, Holland, 2009.

\bibitem{solecki1} S. Solecki, \textit{Decomposing Borel sets and functions and the structure of Baire class 1 functions}, J. Amer. Math. Soc. 11 (1998), 521--550.

\end{thebibliography}
\end{document}